\documentclass[fleqn, eqnumis, pdf]{hjms}

\usepackage{footnote}
\usepackage{multicol}
\usepackage{booktabs}
\usepackage[flushleft]{threeparttable}
\usepackage[font=small,labelfont=bf]{caption}

\renewcommand{\P}{\mathbb{P}}
\newcommand{\E}{\mathbb{E}}
\newcommand{\Rd}{\mathbb{R}^d}
\newcommand{\Rp}{\mathbb{R}^+}
\newcommand{\R}{\mathbb{R}}

\newcommand{\barr}{\begin{array}{rcl}}
\newcommand{\earr}{\end{array}}

\newcommand{\disp}{\displaystyle}
\newcommand{\Ema}{\E^{m_a}}
\newcommand{\Pma}{\P^{m_a}}
\newcommand{\ind}{\textup{1{\hskip -2.5 pt}\hbox{I}}}

\makeatletter
\newcommand{\mathleft}{\@fleqntrue\@mathmargin\parindent}
\newcommand{\mathcenter}{\@fleqnfalse}
\makeatother

\begin{document}

\title{A Multiplier Related to Symmetric Stable Processes}

\author{ Deniz Karl\i\footnote{ I\c{s}\i k University, Department of Mathematics, AMF233,  \c{S}ile Istanbul Turkey 34980\ Email: {\tt deniz.karli@isikun.edu.tr  \& deniz.karli@gmail.com}}\footnote{Partially funded by the BAP project 14B103 of I\c{s}\i k University.}}
{ }

\begin{abstract}
In two recent papers \cite{karli} and \cite{karli2}, we generalized some classical results of Harmonic Analysis using probabilistic approach by means of a $d$-dimensional rotationally symmetric stable process. These results allow one to discuss some boundedness conditions with weaker hypotheses. In this paper, we study a multiplier theorem using these more general results. We consider a product process consisting of a $d$-dimensional symmetric stable process and a 1-dimensional Brownian motion, and use properties of jump processes to obtain bounds on jump terms and the $L^p(\Rd)$-norm of a new operator. 
\end{abstract}



\section{Introduction and Preliminaries}

Classical multiplier problems have been studied for a long time using techniques of  Analysis. For nearly 40 years, a probabilistic approach has been developed which allowed further studies in Operator Theory. The early probabilistic techniques were based on properties of Brownian motion and martingales. In late 1970's and early 1980's this study was generalised to a setup where Markov processes were in focus. Some early research was done by P.A. Meyer. He presented in \cite{meyer,meyer_2,meyer_3} probabilistic techniques which were used to prove boundedness of Littlewood-Paley functions by means of martingales. After the presentation of these new techniques,  boundedness results on $L^p(\Rd)$ were revisited. In early years, the main approach was  to develop alternative  proofs of classical results in the language of probability. Later  some generalisations followed. We recommend the reader  \cite[Chapter IV]{bass} to have a grasp of the probabilistic approach to this theory when Brownian motion is considered as the Markov process of interest.

Together with the boundedness of  operators, multiplier theorems, among which there are Mihlin's and Marcinkiewicz' \cite[IV.5]{bass}, are also widely studied from the probabilistic point of view.  Well-known multiplier theorems became the focus of this research, and it was aimed to determine whether their constraints could be generalised by the tools of Martingale Theory. One of these studies was done by N. Bouleau and D. Lamberton in \cite{bouleau} where the Authors proved boundedness of an operator using a product of a Markov process with a stable process. Their arguments yields a new multiplier theorem with a new set of conditions. The condition is the integrability of a bounded Borel measurable function with respect to a kernel corresponding to their product process. 

Our motivation originates from aforementioned paper \cite{bouleau}. In the their work, N. Bouleau and D. Lamberton worked with a general type of processes. However the tools of the general case are restrictive. In a previous paper \cite{karli}, we focused on a specific process, namely a symmetric stable process, and we replaced Brownian motion of the classical setup with this stable process. By means of these processes, we introduced new types of Littlewood-Paley operators, and proved boundedness results which could not be proved for general Markov processes. These new operators present an opportunity to study multiplier results obtained in \cite{bouleau} once again. Now we have a set of more specific tools using which we aim to prove a new multiplier theorem  in this paper and present an application of our  new boundedness results.

In  \cite{karli} and later in \cite{karli2}, we studied some operators related to classical Littlewood-Paley Theory. For this purpose, we adapted probabilistic approach to the theory. In the classical setup, one uses a $d+1$-dimensional Brownian motion in the upper half space of $\Rd\times\Rp$ and let it run until the process hits the boundary. In our work, we replaced $d$-dimensional component of Brownian motion with a symmetric stable process and kept 1-dimensional Brownian motion as is. Under this new setup, we introduced some new operators, and we proved their boundedness under certain conditions. 

First of all, we introduce some preliminary results from \cite{karli} and state references for details in this section. Throughout the paper $c$ will denote a positive constant. Its value may change from line to line. 

We consider a $d$-dimensional right continuous (rotationally) symmetric $\alpha$-stable process $Y_t$ for $\alpha\in (0,2)$, that is, $Y_t$ is a right continuous Markov process with independent and stationary increments whose characteristic function is $\E(e^{i\xi Y_s})=e^{-s|\xi|^\alpha}, \xi\in\Rd, s>0$. The function $p(s,x,y)$ will denote its  (symmetric) transition density such that 
$$\P^x(Y_s\in A)=\int_A p(s,x,y)\,dy,$$
and $P_s$  will denote the corresponding semi-group $P_s(f)(x)=\E^x(f(Y_s))$. Here $\P^x$ is the probability measure for the process started at $x\in \Rd$, and $\E^x$ is the expectation taken with respect to $\P^x$. A well-known property of symmetric $\alpha-$stable processes is that the transition density $p(s,x,0)$ satisfies the scaling property 
\begin{align}\label{scaling} p(s,x,0)=s^{-d/\alpha}p(1,x/s^{1/\alpha},0), \qquad x\in\Rd,\, s>0,\end{align} 
which can be verified by using the characteristic function.

Similarly, we denote  a one-dimensional Brownian motion (independent from $Y_s$) by $Z_s$ and the probability measure for the process started at $t>0$ by $\P^t$. The process of interest is the product $X_s=(Y_s,Z_s)$ started at $(x,t)\in\Rd\times\Rp$, the corresponding probability measure and the expectation are $\P^{(x,t)}$ and $\E^{(x,t)}$, respectively. Next, define the stopping time $T_0=\inf\{s\geq 0: Z_s=0\}$ which is the first time $X_t$ hits the boundary of $\Rd\times\Rp$. Note that the first exit time $T_0$ depends only on $Z_s$ due to the independence of $Z_s$ and $Y_s$.

To provide a connection between probabilistic and deterministic integrals, we will use two main tools; a new measure $\Pma$ and the vertical Green function.  Denoting the Lebesgue measure on $\Rd$ by $m(\cdot)$, we define the measure $\Pma$ by
$$\Pma=\int_{\Rd} \P^{(x,a)} m(dx),\qquad a>0.$$
Note that this is not a probability measure.
Let $\Ema$ denote the ``expectation" with respect to this measure. We note that for any measurable set $A$ in $\Rd$,
\begin{align}\label{law_of_exit}
	\Pma(X_{T_0}\in A)&=\int\int_0^\infty\int \ind_A(y)p(s,x,y)dy\,\P^t(T_0\in ds)\,m(dx) \nonumber\\
					&=\int_0^\infty\int\int \ind_A(x+y) m(dx)\,p(s,0,y)dy\,\P^t(T_0\in ds)=m(A),
\end{align}
and so the law of $X_{T_0}$ under this new measure is $m(\cdot)$.  Throughout this paper, we will use $dx$ to denote Lebesgue measure on $\Rd$. Just to emphasize the connection between the measure $\Pma$ and the Lebesgue measure, sometimes we will use $m(dx)$ instead of $dx$. But both of them should be treated as the same measure.

Second, for a positive Borel function $f$, the vertical Green function is given by
\begin{align}\label{greens_function}
	\E^a\left[\int_0^{T_0}f(Z_s)\, ds\right]=\int_0^\infty (s\wedge a)f(s)ds.
\end{align}
This property will be used together with the new measure $\Pma$ in Section \ref{multiplier_section}.

 Harmonic functions play a key role in showing boundedness of Littlewood-Paley operators. Here we adapt the probabilistic interpretation of  a harmonic function. A continuous function $u:\Rd\times\Rp\rightarrow \mathbb{R}$ is said to be harmonic (with respect to the process $X_t$) if $u(X_{s\wedge T_0})$ is a martingale with respect to the filtration $\mathcal{F}_s=\sigma(X_{r\wedge T_0}: r\leq s)$ and the probability measure $\P^{(x,t)}$ for any starting point $(x,t)\in\Rd\times\Rp$. One way to obtain such a harmonic function is to start with a bounded Borel function $f:\Rd\rightarrow\mathbb{R}$ and define an extension $u$ by
\begin{align*}
	u(x,t):=\E^{(x,t)}f(Y_{T_0})=\int_0^\infty \E^xf(Y_s)\P^t(T_0\in ds)
\end{align*}
where $\P^t(T_0\in ds)$ is the exit distribution of one-dimensional Brownian motion which is given by
\begin{align*}
	\mu_t(ds):=\P^t(T_0\in ds)=\frac{t}{2\sqrt{\pi}}e^{-t^2/4s}s^{-3/2}ds
\end{align*}
(see \cite{meyer}). By a slight abuse of notation, we will denote both the function on $\Rd$ and its extension to the upper-half space by the same letter, that is, $f_t(x):=f(x,t)=\E^{(x,t)}f(Y_{T_0})$. Next, we define the semi-group $Q_t=\int_0^\infty P_s \mu_t(ds)$.  This semi-group provides us a representation of the extension 
\begin{align*}
	\disp f_t(x)=f(x,t)=Q_tf(x)=\int f(y)\int_0^\infty p(s,x,y)\mu_t(ds)dy.
\end{align*}
We note that this is a convolution with the probability kernel $$q_t(x)=\int_0^\infty p(s,x,0)\mu_t(ds)$$  which is radially decreasing in $x$ (see \cite[Section 3]{karli}) and its Fourier transform is 
\begin{align}\label{q_transform}
	\disp \widehat{q_t}(\cdot)=e^{-t|\cdot|^{\alpha/2}}.
\end{align}

An immediate result is that both of the semi-groups $P_t$ and $Q_t$ are invariant under the Lebesgue measure, that is,
\begin{align}\label{invariance}
	\int P_tf(x)m(dx)=\int f(x)m(dx)=\int Q_tf(x)m(dx).
\end{align}
This follows from the symmetry of the kernels and the invariance under the dual action.

One of the key tools in proving certain inequalities is the density estimates on $p(s,x,0)$. Although there is an infinite series expansion, it is not very easy to work with. For this purpose, we will use the two-sided estimate
  \begin{align}\label{sas_estimate}
 		c_1\,(s^{-d/\alpha} \wedge \frac{s}{|x-y|^{d+\alpha}}) \leq p(s,x,y) \leq c_2\,(s^{-d/\alpha} \wedge \frac{s}{|x-y|^{d+\alpha}}) 
  \end{align}
$(s,x,y)\in\Rp\times\Rd\times\Rd$, which allows us to control the tail of the transition density. In addition, we will need to control the derivative of $p(s,x,0)$. The following Lemma provides this control. Let $\partial^k_{x_j}$ denote the $k^{th}$ partial derivative in the direction of $j^{th}$ coordinate.

\begin{lemma}\label{lemma_der_of_dens}
For $k=1,2$ and $j=1,...,d$
\begin{itemize}
	\item[i.] $\disp \left|\partial_{x_j}^kp(1,x,0)\right|\leq c \left( 1 \wedge \frac{1}{|x|^k}\right)p(1,x,0)$ and \\
	\item[ii.]  $\disp \left|\partial_{x_j}^kp(t,x,0)\right|\leq c \left( t^{-k/\alpha} \wedge \frac{1}{|x|^k}\right)p(t,x,0)$ whenever $t>0$.
\end{itemize}
\end{lemma}
\begin{proof}
	First note that (ii.) follows from (i.) once we use the scaling property 
	\begin{align*}
		p(t,x,0)=t^{-d/\alpha}p(1,x/t^{1/\alpha},0).
	\end{align*} 
	Hence it is enough to prove part (i.). 
	
	The density function $p(1,x,0)$ is obtained by Fourier inversion of $e^{-|x|^\alpha}$. Moreover, the function  $x^\beta e^{-|x|^\alpha}$ is integrable for any multi-index $\beta$. Hence the density function and all of its partial derivatives are bounded.
	
	Another way to obtain the density function is to use $g_{\alpha/2}$, the density of an $\alpha/2$ stable subordinator whose Laplace transform is given by $\int_0^\infty e^{-\lambda v} g_{\alpha/2}(t,v)dv=e^{-t\lambda^{\alpha/2}}$. Using $g_{\alpha/2}$, we can express the density as
	\begin{align}\label{dens_rep}
		p(1,x,0)=\int_0^\infty\frac{1}{(4\pi s)^{d/2}} e^{-|x|^2/(4s)}g_{\alpha/2}(1,s)ds.
	\end{align}
	This can be verified by taking the Fourier transform of the Equation (\ref{dens_rep}). For some constant $M>0$, and $|x|>M$, we can obtain the first partial derivative in $x_1$,
	\begin{align}\label{der_of_dens}
		\disp \partial^1_{x_1}p(1,x,0)=-2\pi x_1\int\frac{1}{(4\pi s)^{(d+2)/2}} e^{-|x|^2/(4s)}g_{\alpha/2}(1,s)ds.
	\end{align}
	Temporarily, let us denote the $d$-dimensional density function by $p^{(d)}(1,x,0)$ and define $\tilde x=(x,0,0)\in{\mathbb{R}^{d+2}}$ whenever $x\in\Rd$. The equation (\ref{der_of_dens}) can be written as
	$$\partial^1_{x_1}p^{(d)}(1,x,0)=-2\pi x_1 p^{(d+2)}(1,\tilde x,0)$$
	and hence using the upper-bound in (\ref{sas_estimate})
	$$\left|\partial^1_{x_1}p^{(d)}(1,x,0)\right|\leq \frac{c}{|x|^{d+\alpha+1}} $$
	which gives the result for $k=1$. The case $k=2$ is similar when we observe that
	$$\partial^2_{x_1}p^{(d)}(1,x,0)=-2\pi p^{(d+2)}(1,\tilde x,0)+4\pi^2x^2_1p^{(d+4)}(1,\overline{x},0)$$
	where $\overline{x}=(\tilde x ,0,0)\in\mathbb{R}^{d+4}$. 
\end{proof}

\section{A Bound on Jump Terms}

In this section, we will discuss some preliminary results on our jump process. One can find the details for general jump processes in \cite{bass_stochastic_pr}. First, let us define a new martingale by $M_t=f(X_{t\wedge T_0})$ where $X_t$ is the product process explained in the previous section and $f$ is a bounded measurable function. $M_t$ is a jump process and, moreover, $M_t$ has at most countably many jumps (see \cite[Ch. 17.1, p 132]{bass_stochastic_pr} for the details). So we can label these jumps as $T_1, T_2, ...$ so that each $T_i$ is either predictable or inaccessible with disjoint graphs. Moreover, the process $M_t$ has jumps only at these times $T_i$ and jump sizes are bounded. Here we do not assume that $T_i$'s are ordered, that is, $i<j$ does not imply $T_i<T_j$.

Now consider the process
$$A_t^n=\Delta M_{T_n} \ind_{\{t \geq T_n\}} \ind_{\{|\Delta Y_{T_n}|<|Z_{T_n}|^{2/\alpha}\}}.$$
We note that this process makes a jump of size $\Delta M$ at time $T_n$ and constant afterwards if the jumps size is small, that is, if the jump size $|\Delta Y_{T_n}|<|Z_{T_n}|^{2/\alpha}$. (In this paper, a jump at time $T_n$ is referred as a large jump, if its size is $|\Delta Y_{T_n}|\geq|Z_{T_n}|^{2/\alpha}$.) Since the vertical component  $Z_t$ of $X_t$ is continuous, the jumps correspond to those of the horizontal component $Y_t$. So this process ignores large jumps of $M_t$.

Since $f$ is real-valued, we can write $A_t^n$ as a difference of two increasing processes, one with positive jumps and the other with negative jumps. Hence $A_t^n$ admits a compensator, say $\tilde A_t^n$. So the process $M_t^n$ obtained as 
$$M_t^n=A_t^n-\tilde A_t^n$$
is a martingale. Note that $M-M^n$ has no small jumps at time $T_n$. Moreover, $M^n$'s are pairwise orthogonal, since each of them correspond to jumps at different times. In other words,
$$\Delta M_t^n\Delta M_t^m=0$$
whenever $n\not=m$ for all $t>0$. Then by the orthogonality Lemma \cite[Lemma 17.2]{bass_stochastic_pr}, we obtain
$$\E^{(x,a)}(M_\infty^nM_\infty^m)=0$$
and also
$$\E^{(x,a)}\left(M_\infty^n(M_\infty-\sum_{i=1}^mM_\infty^i)\right)=0$$
for $n\leq m.$ Then using this orthogonality property, we can write
\begin{align*}
	\E^{(x,a)}\left[(M_\infty)^2\right]&=\E^{(x,a)}\left[(M_\infty-\sum_{i=1}^mM_\infty^i+\sum_{i=1}^mM_\infty^i)^2\right]\\
							&=\E^{(x,a)}\left[(M_\infty-\sum_{i=1}^mM_\infty^i)^2+(\sum_{i=1}^mM_\infty^i)^2\right]\\
							&\geq\E^{(x,a)}\left[(\sum_{i=1}^mM_\infty^i)^2\right]\\
							&=\sum_{i=1}^m\E^{(x,a)}\left[(M_\infty^i)^2\right].
\end{align*}
Hence we see
\begin{align*}
	\sum_{i=1}^m\E^{(x,a)}\left[(M_\infty^i)^2\right]&= \E^{(x,a)}\left[(\sum_{i=1}^mM_\infty^i)^2\right]\leq \E^{(x,a)}\left[(M_\infty)^2\right]<\infty
\end{align*}
for any positive integer $m$, and so the series $ \sum_{i=1}^m\E^{(x,a)}\left[(M_\infty^i)^2\right]$ converges. We conclude that $\sum_{i=1}^m M_t^i$ converges in $L^2(\P^{(x,a)})$, say to $U_t$. So $U_t$ is the purely discontinuous part of $M_t$ consisting of small jumps. Hence $[U]_t=\sum_{0\leq s \leq t} (\Delta U_s)^2$. Denoting the continuous part by $M^c_t$, we write 
$$M_t=M_t^c+U_t+\tilde U_t$$
where $\tilde U_t$ is the part corresponding to large jumps. Since $Z_t$ is Brownian Motion, two processes $\Delta U_t$ and
$$\ind_{\{t < T_0\}} (f(Y_t,Z_t)-f(Y_{t-},Z_t)) \ind_{\{|Y_t - Y_{t-}|<|Z_{t}|^{2/\alpha}\}}$$
are indistinguishable. Moreover,
$$[M]_t=\langle M^c\rangle_t + \sum_{0\leq s \leq t} (\Delta M_s)^2$$
and 
$$[M]_\infty\geq [U]_\infty.$$
Then by the Burkholder-Davis-Gundy Inequality
\begin{align*}
	\Ema |U_t|^p &=\int \E^{(x,a)} |U_t|^p m(dx)\\
			&\leq \int \E^{(x,a)} \left( \sup_{0\leq s \leq t} |U_s|^p\right) m(dx)\\
			&\leq c \int \E^{(x,a)} [U]_t^{p/2}m(dx)
\end{align*}
and so 
\begin{align*}
	\Ema |U_\infty|^p &\leq c \Ema [U]_\infty^{p/2}\leq c \Ema [M]_\infty^{p/2}.
\end{align*}
After using the Burkholder-Davis-Gundy Inequality once again and applying Doob's Inequality, we have
\begin{align*}
	 \Ema [M]_t^{p/2}&\leq c\, \Ema \left( \sup_{0\leq s \leq t} |M_s|^p\right)\leq c\, \Ema |M_t|^{p}
\end{align*}
which leads to
\begin{align*}
	 \Ema [M]_\infty^{p/2}\leq c\, \Ema |M_\infty|^{p}
\end{align*}
and hence by (\ref{law_of_exit})
\begin{align}\label{U_f_relation}
	 \Ema |U_\infty|^{p}\leq c\, \Ema |M_\infty|^{p}&=c\,\int|f(x)|^p\Pma(X_{T_0}\in dx)\\
										&=c\,\int|f(x)|^p m(dx)\\
										&=c\, \|f\|_p^p.
\end{align}

\section{A Multiplier Theorem}\label{multiplier_section}

Let $r(t)$ be a bounded Borel function on $\Rp$ and $\alpha \in (0,1)$. Define 

\begin{align}\label{op_T}
   Tf(x) &= \int_0^\infty t \cdot r(t) \, \int _{|h|<t^{2/\alpha}} [-(Q_t\tau_{h/2}-Q_t\tau_{-h/2})^2](f)(x) \frac{dh}{|h|^{d+\alpha}} dt 
\end{align}
where $\tau_h$ is the translation operator given by $$\tau_hf(x)=f(x+h)$$ and $Q_tf$ is the extension of $f$ to upper-half space $\Rd\times\Rp$ as described before, that is, $$Q_tf(x)=f_t(x)=f(x,t)=\E^{(x,t)}f(Y_{T_0}).$$ First we should note that the semigroup $Q_t$ commutes with the translation operator. This follows from the commutative property of the semigroup $P_s$ of the L\'evy Process $Y_s$ with $\tau_h$. To be clear,
\begin{align*}
    Q_t\tau_hf(x)&=\int_0^\infty P_s \tau_h f(x) \mu_t(ds) =\int_0^\infty \tau_hP_s  f(x) \mu_t(ds) \\
		&=\int_0^\infty P_s  f(x+h) \mu_t(ds) =\tau_hQ_tf(x),
\end{align*}  
for any $x\in\Rd$. Hence by using the commutativity of operators $Q_t$ and $\tau_h$ and the semigroup property of $Q_t$, we can write 
 \begin{align*}
    -(Q_t\tau_{h/2}-Q_t\tau_{-h/2})^2f&=(2Q_{2t}-Q_{2t}\tau_h-Q_{2t}\tau_{-h})f\\
				&=2f_{2t}(x)-f_{2t}(x+h)-f_{2t}(x-h).
 \end{align*}
In this form, it is easy to see that $T$ is linear, that is, for any $f,g\in L^p$ and $c\in\R$, we have $T(cf+g)=cTf+Tg$.

Second, we will show that this operator is well-defined on $C_c^\infty(\Rd)$, the space of smooth and compactly supported functions, which is dense in $L^p(\Rd)$. For this purpose, let $f\in C_c^\infty(\Rd)$. Since $r$ is a bounded function and $$2f_{2t}(x)-f_{2t}(x+h)-f_{2t}(x-h)=\left(f_{2t}(x)-f_{2t}(x+h)\right)+\left(f_{2t}(x)-f_{2t}(x-h)\right),$$ it is enough to consider the following two integrals:
\begin{align*}
   I_1 &= \int_0^1 t  \, \int _{|h|<t^{2/\alpha}} \left(f_{2t}(x+h)-f_{2t}(x)\right) \frac{dh}{|h|^{d+\alpha}} dt 
\end{align*}
and
\begin{align*}
   I_2 &= \int_1^\infty t  \, \int _{|h|<t^{2/\alpha}}  \left(f_{2t}(x+h)-f_{2t}(x)\right)  \frac{dh}{|h|^{d+\alpha}} dt. 
\end{align*}
To bound the first integral, note that 
\begin{align*}
	 \left|f_{2t}(x+h)-f_{2t}(x)\right| &\leq\int_0^\infty \int |f(y+h)-f(y)| p(s,x,y)dy\,\mu_{2t}(ds)\\
							&\leq c \|\nabla f\|_\infty |h|.
\end{align*}
Then 
\begin{align}\label{first_int_bound}
	 \left|I_1\right| &\leq c \|\nabla f\|_\infty \int_0^1 t  \, \int _{|h|<t^{2/\alpha}}  \frac{dh}{|h|^{d+\alpha-1}} dt\leq c \|\nabla f\|_\infty.
\end{align}
To bound the second integral we use Lemma \ref{lemma_der_of_dens}. We observe that
\begin{align*}
	|p(s,x+h,y)-p(s,x,y)|\leq c |h| s^{-1/\alpha}p(s,\xi,y)\leq c |h| s^{-1/\alpha}s^{-d/\alpha}.
\end{align*}
Hence
\begin{align*}
	 \left|f_{2t}(x+h)-f_{2t}(x)\right| &\leq\int_0^\infty \int |f(y)| |p(s,x+h,y)-p(s,x,y)|dy\,\mu_{2t}(ds)\\
							&\leq c \| f\|_1 |h|\int_0^\infty s^{-(d+1)/\alpha}\mu_{2t}(ds)\\
							&\leq c \| f\|_1 |h| \, t^{-(2d+2)/\alpha}.
\end{align*}
Then
\begin{align}\label{second_int_bound}
	 \left|I_2\right| &\leq c \| f\|_1 \int_1^\infty t\cdot  t^{-(2d+2)/\alpha}  \, \int _{|h|<t^{2/\alpha}}  \frac{dh}{|h|^{d+\alpha-1}} dt\leq c \|\ f\|_1.
\end{align}
Hence $T$ is well-defined on $C_c^\infty(\Rd)$.

We note that the above are finite if $\alpha\in (0,1)$. Although the tools on the jump processes in section 1 works whenever $\alpha \in (0,2)$, the operator $T$ is well-defined if $\alpha<1$. Hence we restrict our result to the case $\alpha\in (0,1)$.

Our main result is that $T$ is a bounded operator on $L^p$.

\begin{theorem}
$T$, as defined in (\ref{op_T}), is a bounded operator on $L^p(\Rd)$ for $p>1$ and $\alpha\in (0,1)$.
\end{theorem}
\begin{proof}
Let $f\in C_c^\infty(\Rd)$, $r$ be a bounded Borel function and $T$ be as defined in (\ref{op_T}). Without loss of generality we may assume $r$ is positive valued, since otherwise we can write this function as a combination of positive and negative parts, $r^+-r^-$, and apply the argument to both $r^+$ and $r^-$. 

We prove the statement in four steps.

In the first step, we will use the purely discontinuous martingale $U_t$ which was defined in the previous section. By means of $U_t$, let us define a new martingale 
$$V_t=\int_0^t r(Z_s)dU_s$$
where $Z_s$ is the 1-dimensional Brownian component of $X_s$.
By considering
\begin{align*}
   \Delta V_t&=\int_{t-}^{t} r(Z_s) dU_s=r(Z_t) \Delta U_t \\
				&= \ind_{\{0 \leq t \leq T_0\}} \ind_{\{|\Delta Y_t|<|Z_s|^{2/\alpha}\}} (f(Y_t,Z_t)-f(Y_{t-},Z_t))\cdot r(Z_t)
\end{align*}
and using Burkholder-Davis-Gundy inequality, we can bound the expectation of this martingale. That is,
\begin{align}\label{v-infty}
   \Ema|V_\infty|^p &=\int\E^{(x,a)}|V_\infty|^p m(dx)\leq c \, \int \E^{(x,a)} [V]_\infty^{p/2} m(dx)\\
			&\leq c \, \int \E^{(x,a)} [U]_\infty^{p/2} m(dx)\leq c \, \int \E^{(x,a)} |U_\infty|^{p} m(dx)\leq c \, ||f||^p_p \nonumber
\end{align}
by the Inequality (\ref{U_f_relation}).

Now let $q$ be the conjugate exponent of $p$ so that $(1/p)+(1/q)=1$ and $g$ be a compactly supported smooth function in $L^q(\Rd)\cap L^2(\Rd)$. Similar to before, we denote the extension of $g$ by $Q_tg(x)$ or by $g(x,t)$. We also denote the martingale which is obtained from $g$ by $N_t=g(X_{t\wedge T_0})$.

By the product rule for martingales with jumps (see \cite[pp.231]{applebaum}) we obtain
\begin{align*}
  \Ema(V_\infty N_\infty) &= \int\E^{(x,a)}(V_\infty N_\infty) m(dx)=\int \E^{(x,a)} \left(\sum_{t>0} \Delta V_t \Delta N_t\right) m(dx). 
\end{align*}
Note that 
\begin{align*}
   \Delta V_t \Delta N_t= \ind_{\{0 \leq t \leq T_0\}}\ind_{\{|\Delta Y_t|<|Z_s|^{2/\alpha}\}} r(Z_t)\cdot (f(Y_t,Z_t)-f(Y_{t-},Z_t))\\\cdot (g(Y_t,Z_t)-g(Y_{t-},Z_t)). 
\end{align*}
Throughout the rest of this section, we will denote the above function by $\Lambda$ for the sake of simplicity. We write for any $(x,y,t)\in\Rd\times\Rd\times\Rp$
\begin{align*}
   \Lambda(x,y,t)=  \ind_{\{|y-x|<|t|^{2/\alpha}\}}r(t)\cdot (f(y,t)-f(x,t))\cdot (g(y,t)-g(x,t)). 
\end{align*}
Then we have 
\begin{align}\label{main1}
  \Ema(V_\infty N_\infty) &= \Ema \left(\sum_{0\leq s\leq T_0} \Lambda(Y_{s-},Y_s,Z_s) \right). 
\end{align}
If we use the L\'evy system formula \cite[Theorem 1]{karli} here, we obtain
\begin{align*}
    \Ema(\sum_{0\leq s \leq T_0} \Lambda(Y_{s-},Y_s,Z_s)) = \int \E^{(x,a)}\left(\sum_{0\leq s \leq T_0} \Lambda(Y_{s-},Y_s,Z_s)\right) m(dx)\\ \\
			= c \, \int \E^{(x,a)}\left(\int_0^{T_0} \int \Lambda(Y_{s},Y_s+h,Z_s) \,\frac{dh}{|h|^{d+\alpha}} \, ds  \right)m(dx). 
\end{align*} 
The semigroup $P_t$ is invariant under the measure $m$. Hence the last line equals
\begin{align}\label{int_green}
   c \,  \int \E^a\left(\int_0^{T_0} \int \Lambda(x,x+h,Z_s) \,\frac{dh}{|h|^{d+\alpha}}\, ds  \right)\, m(dx).
\end{align} 
Next, if we set 
$$\beta_x(t)= \int \Lambda(x,x+h,t) \,\frac{dh}{|h|^{d+\alpha}}$$
and write $\beta_x=\beta_x^+-\beta^-$ where $\beta_x^+$ and $\beta_x^-$ are positive and negative parts of the function, respectively, then we can apply the Green Kernel of the vertical component (i.e., the Brownian component) to the integrals
$$\E^a\left(\int_0^{T_0} \beta_x^+(Z_s)\, ds  \right) - \E^a\left(\int_0^{T_0}\beta_x^-(Z_s) \, ds  \right).$$
Then the integral in (\ref{int_green}) becomes
\begin{align}\label{main2}
  \lefteqn{ \Ema(\sum_{0\leq s \leq T_0} \Lambda(Y_{s-},Y_s,Z_s)) }\\&\qquad= c \,  \int \int_0^{\infty} (t\wedge a) \int \Lambda(x,x+h,t) \,\frac{dh}{|h|^{d+\alpha}} \,dt  \, m(dx). 
\end{align} 

In the second step, we will express $Tf$ in terms of the above integral. For this purpose, we will denote the usual inner product by $\langle\cdot,\cdot\rangle$ and write
$$\langle f,g\rangle=\int f(x)g(x)m(dx).$$
If $g$ is as taken before, we can write
\begin{align}\label{inner_product}
	 \lefteqn{ \langle Tf,g\rangle=} \\&\int g(x)\int_0^\infty t \, r(t) \int_{|h|<t^{2/\alpha}} \left( Q_t\tau_{h/2}-Q_t\tau_{-h/2}\right)^2(f)(x) \frac{dh}{|h|^{d+\alpha}}\,dt\, m(dx).
\end{align}
By the Inequalities (\ref{first_int_bound}) and (\ref{second_int_bound}) and using the assumption that $g$ is compactly supported and continuous, we conclude that this inner product is bounded, that is,
\begin{align*}
	|\langle Tf,g\rangle |\leq c \|r\|_\infty (\|\nabla f\|_\infty+\|f\|_1)\|g\|_1<\infty.
\end{align*}
Hence one can use Fubini to interchange the order of integrals in (\ref{inner_product}) to write
\begin{align*}
	\langle Tf,g\rangle&=\int_0^\infty t \, r(t) \int_{|h|<t^{2/\alpha}} \langle \left( Q_t\tau_{h/2}-Q_t\tau_{-h/2}\right)^2(f),g\rangle \frac{dh}{|h|^{d+\alpha}}\,dt\,.
\end{align*}
At this point, as a side note, we observe that this inner product relates to $\overrightarrow{G}_{f,\alpha}$ which was defined in \cite{karli} by
\begin{align*}
	\overrightarrow{G}_{f,\alpha}(x)  &  =\disp\left[ \int_0^\infty t\, \int_{|h|<t^{2/\alpha}} [f_t(x+h)-f_t(x)]^2\frac{dh}{|h|^{d+\alpha}} \,dt  \right]^{1/2} .
\end{align*}

 To see this, we use commutativity of $\tau_h$ and $Q_t$, symmetry of $Q_t$ and linearity of the inner product.

\begin{align*}
	 \lefteqn{< -(Q_t\tau_{h/2}-Q_t\tau_{-h/2})^2 f , g >}\qquad\\& =< 2 Q_{2t} f-  Q_{2t} \tau_{-h} f -Q_{2t}\tau_h f ,  g >\\
		&= <  Q_{2t} f-  Q_{2t} \tau_{-h} f ,  g >-<  Q_{2t}\tau_h f -Q_{2t} f,  g >\\
		& =  <  \tau_{-h}Q_{2t}\tau_h f-\tau_{-h}Q_{2t} f,  g> - <  Q_{2t}\tau_h f-Q_{2t} f,  g> \\
		& = <  Q_{2t}\tau_h f-Q_{2t} f, \tau_h g> - <  Q_{2t}\tau_h f-Q_{2t} f,  g> \\
		&= <  Q_{2t}\tau_h f-Q_{2t} f , \tau_h g- g > \\
		&= <  Q_t(Q_t\tau_h f-Q_t f) , \tau_h g- g > \\
		&= <  Q_t\tau_h f-Q_t f , Q_t(\tau_h g- g)  > \\ 
		& = <  Q_t\tau_h f-Q_t f , Q_t\tau_h g-Q_t g  > \\
		&=\int    (Q_tf(x+h)-Q_tf(x))(Q_tg(x+h)-Q_tg(x)) m(dx)\\
		&=\int    (f_t(x+h)-f_t(x))(g_t(x+h)-g_t(x)) m(dx).
\end{align*}
This equality also shows that
\begin{align}\label{main3}
	\lefteqn{ \langle Tf,g\rangle= \int \int_0^{\infty} t \int \Lambda(x,x+h,t) \,\frac{dh}{|h|^{d+\alpha}} \,dt  \, m(dx)}\\
					&= \int \int_0^{\infty} t \int_{|h|<t^{2/\alpha}} (f_t(x+h)-f_t(x))(g_t(x+h)-g_t(x)) \,\frac{dh}{|h|^{d+\alpha}} \,dt  \, m(dx).\nonumber
\end{align}

In the third step, we will find a bound for the inner product above. For this purpose, we apply the H\"older Inequality first to get
\begin{align*}
 | \Ema(V_\infty N_\infty)| &\leq  \Ema(|V_\infty|^p)^{1/p}\cdot  \Ema(| N_\infty|^q)^{1/q}. 
\end{align*}
By (\ref{v-infty}), we obtain $ \Ema(|V_\infty|^p)^{1/p}\leq c \|f\|_p$ and also we have $ \Ema(| N_\infty|^q)^{1/q}=\|g\|_q$  by (\ref{law_of_exit}) . Hence
\begin{align}\label{main3}
  \left|  \int \int_0^{\infty} (t\wedge a) \int \Lambda(x,x+h,t) \,\frac{dh}{|h|^{d+\alpha}} \,dt  \, m(dx)\right|\leq c \|f\|_p\|g\|_q 
\end{align} 
for any $a>0$, by (\ref{main1}) and (\ref{main2}). Now if we let $a\rightarrow \infty$, then we obtain
\begin{align}\label{main4}
	| \langle Tf,g\rangle |&\leq c \|f\|_p\|g\|_q.
\end{align}
Hence $T$ is a bounded operator on $C_c^\infty(\Rd)$ and it extends to $L^p(\Rd)$.
\end{proof}

\begin{corollary}
The function 
  $$
      m(\xi)=c\, |\xi|^{\alpha} \, \int_0^\infty t\cdot r(t) \, e^{-2t|\xi|^{\alpha/2}} \,dt
  $$
is an $L^p(\Rd)$-multiplier for $p>1$ and $\alpha\in (0,1)$.

\end{corollary}

\begin{proof}
In the theorem above, we proved that the operator $T$ is bounded on $L^p(\Rd)$. The multiplier of interest here is the multiplier corresponding to this operator. To see this, we take the Fourier transform of $Tf$ and use Fubini's theorem.
\mathleft
 \begin{align*}
	 \lefteqn{\widehat{(Tf)}(\xi) =\int Tf(x)\,e^{i\xi x} m(dx)}\\
                     &= \int e^{i\xi x} \int_0^\infty t \cdot r(t) \int \left[ -(Q_t\tau_{h/2}-Q_t\tau_{-h/2})^2 \right] (f)(x) \frac{dh}{|h|^{d+\alpha}} \, dt \, m(dx)\\
                     &= \int_0^\infty t \cdot r(t) \int \int e^{i\xi x} \left[ -(Q_t\tau_{h/2}-Q_t\tau_{-h/2})^2 \right] (f)(x)\, m(dx)\, \frac{dh}{|h|^{d+\alpha}} \, dt \\
                     &= \int_0^\infty t \cdot r(t) \int \left( \left[ -(Q_t\tau_{h/2}-Q_t\tau_{-h/2})^2 \right] (f) \right)\hat{}(\xi) \, \frac{dh}{|h|^{d+\alpha}} \, dt \\
                     &= \int_0^\infty t \cdot r(t) \int \left( \left[ -Q_{2t}\tau_{h}-Q_{2t}\tau_{-h}+2Q_{2t} \right]  \right)\hat{}(\xi) \, \hat{f}(\xi)\frac{dh}{|h|^{d+\alpha}} \, dt \\
                     &=  \hat{f}(\xi)\int_0^\infty t \cdot r(t) \int \left[ -(Q_{2t})\hat{}(\xi) \, e^{-i\xi h}-(Q_{2t})\hat{}(\xi) \, e^{i\xi h}+2(Q_{2t})\hat{}(\xi)  \right]   \,\frac{dh}{|h|^{d+\alpha}} \, dt. \\
 \end{align*}
\mathcenter
As stated in (\ref{q_transform}), $Q_t$ has the Fourier transform $e^{-t|\xi|^{\alpha/2}}$. Hence we have the above integral equals
 \begin{align*}
                     &=  \hat{f}(\xi) \int_0^\infty t \cdot r(t) \int \left[ -e^{-i\xi h}- e^{i\xi h}+2  \right]  e^{-2t|\xi|^{\alpha/2}}  \, \frac{dh}{|h|^{d+\alpha}} \, dt \\
                     &=  \hat{f}(\xi) \int_0^\infty t \cdot r(t) \, e^{-2t|\xi|^{\alpha/2}}\, \int \left[ -(e^{i\xi h/2}- e^{-i\xi h/2})^2  \right]    \, \frac{dh}{|h|^{d+\alpha}} \, dt \\
                     &=  \hat{f}(\xi) \int_0^\infty t \cdot r(t) \, e^{-2t|\xi|^{\alpha/2}}\, \int 4 \,\sin^2(\xi h/2)    \, \frac{dh}{|h|^{d+\alpha}} \, dt \\
                     &=  \hat{f}(\xi) \int_0^\infty t \cdot r(t) \, e^{-2t|\xi|^{\alpha/2}}\, \int 4 \,\sin^2(|\xi| uh/2)    \, \frac{dh}{|h|^{d+\alpha}} \, dt \\
 \end{align*}
where $u$ is the unit vector in the direction of $\xi$. Then the last line equals
 \begin{align*}
                    \left[ 2^{2-\alpha}  \int  \,\sin^2(uh)    \, \frac{dh}{|h|^{d+\alpha}}\right]  \,\hat{f}(\xi)\, |\xi|^{\alpha}  \int_0^\infty t \cdot r(t) \, e^{-2t|\xi|^{\alpha/2}}\, dt. 
 \end{align*}
  The first integral is finite for any $\alpha \in (0,2)$. Hence the last line is equal to
 $$
     c\,\hat{f}(\xi)\,|\xi|^{\alpha}  \int_0^\infty t \cdot r(t) \, e^{-2t|\xi|^{\alpha/2}}\, dt,
 $$
 which finishes our calculation.

\end{proof}

\subsection*{Acknowledgements}
This research project is supported by the BAP grant, numbered 14B103, at the I\c{s}\i k University, Istanbul, Turkey.

\end{document}